\documentclass[11pt]
{amsart}
\usepackage{amssymb,amsmath,amsthm,amsfonts,amsopn,url,color,hyperref,enumitem,microtype}
\usepackage[all]{xy}

\theoremstyle{plain}
\newtheorem{thm}{Theorem}[section]
\newtheorem*{mainthm}{Main Results}
\newtheorem{prop}[thm]{Proposition}
\newtheorem{lemma}[thm]{Lemma}
\newtheorem{cor}[thm]{Corollary}

\theoremstyle{definition}
\newtheorem{defn}[thm]{Definition}
\newtheorem*{defn*}{Definition}
\newtheorem*{question*}{Question}

\newtheorem*{example*}{Example}
\newtheorem{rem}[thm]{Remark}
\newtheorem*{rem*}{Remark}
\newcommand{\field}[1]{\mathbb{#1}}
\newcommand{\N}{\field{N}}

\newcommand{\ideal}[1]{\mathfrak{#1}}
\newcommand{\m}{\ideal{m}}
\newcommand{\n}{\ideal{n}}
\newcommand{\p}{\ideal{p}}

\newcommand{\func}[1]{\mathrm{#1} \,}

\newcommand{\Spec}{\func{Spec}}

\DeclareMathOperator{\ann}{ann}

\DeclareMathOperator{\Hom}{Hom}

\newcommand{\be}{\begin{enumerate}}
\newcommand{\ee}{\end{enumerate}}

\newcommand{\li}
 {\leftfootline}

\renewcommand{\phi}{\varphi}

\DeclareMathOperator{\Soc}{Soc}

\DeclareMathOperator{\orc}{c}

\author{Neil Epstein}
\address{Department of Mathematical Sciences \\ George Mason University \\ Fairfax, VA  22030}
\email{nepstei2@gmu.edu}

\author{Jay Shapiro}
\address{Department of Mathematical Sciences \\ George Mason University \\ Fairfax, VA  22030}
\email{jshapiro@gmu.edu}

\title[Gaussian elements]{Gaussian elements of a semicontent algebra}

\date{August 24, 2017}
\begin{document}
\begin{abstract}
The connection between a univariate polynomial having locally principal content and the content function acting like a homomorphism (the so-called Gaussian property) has been explored by many authors.  In this work, we extend several such results to the contexts of multivariate polynomials, power series over a Noetherian ring, and base change of affine $K$-algebras by separable algebraically closed field extensions.  We do so by using the framework of the Ohm-Rush content function.   The correspondence is particularly strong in cases where the base ring is approximately Gorenstein or the element of the target ring is regular.
\end{abstract}

\maketitle

\section{Introduction}\label{sec:intro}
Gauss's lemma is fundamental in number theory and algebra. If $R$ is a commutative ring and $f \in R[x]$, the \emph{content} $\orc(f)$ of $f$ is the ideal in $R$ generated by the $R$-coefficients of $f$.  Gauss showed that when $R=\mathbb Z$, we always have \begin{equation}\label{eq:Gauss}
\orc(f)\orc(g) = \orc(fg).
\end{equation}  Along with various avatars, it may be used to explore unique factorization and primitivity of polynomials.  However, in some sense it almost never holds.  In fact \cite{Ts-Gauss}, for an integral domain $R$, (\ref{eq:Gauss}) holds for \emph{all} pairs of polynomials $f,g \in R[X]$ if and only if $R$ is a Pr\"ufer domain (a condition trivially satisfied by $\mathbb Z$).  From this perspective, it makes sense to ask for properties of polynomials that satisfy Gauss's lemma.  One says that $f\in R[X]$ is \emph{Gaussian} if for all $g\in R[X]$, (\ref{eq:Gauss}) holds.  It turns out that even this condition is close to the Pr\"ufer condition.  For a reduced ring $R$ \cite{Lu-Gaussrp}, and also for an approximately Gorenstein Noetherian ring $R$ \cite{HeiHu-Gauss}, $f\in R[X]$ is Gaussian if and only if $c(f)$ is \emph{locally principal}.  Moreover, for any commutative ring $R$ \cite{Lu-Gaussinv}, a regular element $f\in R[X]$ is Gaussian if and only if $c(f)$ is locally principal.

In \cite{OhmRu-content, Ru-content, Nas-zdcontent, Nas-ABconj}, conditions were established and developed on a ring extension such that the notion of the content of an element of the target ring $S$ as an ideal in the base ring $R$ is a useful construct.  As such, the Gaussian property of an element of $S$ with respect to $R$ makes sense, and one would hope to come to similar conclusions in this expanded context.  Examples of ``semicontent algebras'' (defined in \cite{nmeSh-OR} as a generalization of ``content algebras'' \cite{OhmRu-content}) include \begin{itemize}
\item affine semigroup extensions \cite{No-content}, including polynomial extensions in several variables,
\item power series extensions of Noetherian rings \cite{nmeSh-DMpower}, and
\item base change of affine $K$-algebras ($K$ a field) by very well-behaved field extensions $L/K$ \cite[Propositions 3.8 and 3.11]{nmeSh-OR2}.
\end{itemize}
In the first two cases above, the content of an element is the ideal generated by the coefficients of the polynomial-analogue of the extension with respect to the base ring.  In the third case, there's a more interesting answer based on the identity of the target ring as a free module over the base ring.  See the discussion following Proposition~\ref{pr:fieldext}.

The core results of this paper (see Section~\ref{sec:Gausslp}) generalize the results of \cite{HeiHu-Gauss} to the more general context of the Ohm-Rush content function.   However the casual reader who is less well versed in Ohm-Rush content theory may find the corollaries to these theorems (see Section~\ref{sec:apps}), which are stated in a more specific context, to be of primary interest.  In particular, the following is an incomplete representation of what we have proved, being a proper subset of the corollaries of our core results (see Theorems~\ref{thm:appGorMonoid}--\ref{thm:appGorpower} and Corollaries~\ref{cor:fieldextGor}--\ref{cor:fieldextGauss}):

\begin{mainthm}
Let $R$ be a commutative ring, and let $S$ be an $R$-algebra and $f\in S$ as below.  Then under any of the following conditions, $\orc(f)$ is locally principal if and only if Equation~$(\ref{eq:Gauss})$ is satisfied for every $g\in S$: \begin{enumerate}[label=\emph{\Roman*}.]
\item $R$ is Noetherian and approximately Gorenstein, and $S = R[x_1, \ldots, x_n]$ or any other affine semigroup algebra over $R$.
\item $R$ is locally Noetherian, $S = R[x_1, \ldots, x_n]$ or any other affine semigroup algebra over $R$, and $f$ is regular.
\item $R$ is Noetherian and approximately Gorenstein, and $S=R[\![x_1, \ldots, x_n]\!]$.
\item $R$ is a finitely generated artinian Gorenstein $K$-algebra, where $K$ is a field, $L=K(y_1, \ldots, y_t)$ is a purely transcendental field extension, and $S = R \otimes_K L$, where ``content'' is with respect to the field variables $y_j$.
\item $R$ is a finitely generated $K$-algebra, where $K$ is an algebraically closed field, $L/K$ is any field extension, $f$ is regular, and $S=R \otimes_K L$, where ``content'' is with respect to a vector space basis of $L$ over $K$.
\end{enumerate}
\end{mainthm}

The paper is structured as follows: In $\S\ref{sec:OR}$, we recall the framework built up so far about the Ohm-Rush content function.  This allows us a language in which to develop our core results.

In $\S\ref{sec:Gausslp}$, we build Ohm-Rush theory up a bit more, in service of the core theorems of the paper.  What we investigate in this section is the connection between locally principal content and Gaussianness.  One direction (Theorem~\ref{thm:invGauss}) is relatively easy. This first core theorem says that in semicontent algebras, an element of the algebra that has locally principal content will always be Gaussian.  The second of the core theorems (Theorem~\ref{thm:appGor}) gives a partial converse.  It says that if $R$ is approximately Gorenstein, and $S$ is semicontent over $R$, then any Gaussian element of $S$ has locally principal content in $R$.  Then in the third core theorem, Theorem~\ref{thm:localNoeth}, we weaken the hypothesis on $R$ (just locally Noetherian) but strengthen the hypotheses on the $R$-algebra $S$ (requiring it to be free as an $R$-module).  We show that in this case, any \emph{regular} Gaussian element of $S$ has locally principal content ideal.

In $\S\ref{sec:Pr}$, we obtain in Theorem~\ref{thm:Pr} a new characterization of Pr\"ufer domains based on the Gaussian property.  Finally in $\S\ref{sec:apps}$, we apply the built up theorems to draw connections between Gaussian elements and locally principal  content ideals for power-series extensions, multivariate polynomial extensions, and base change through certain kinds of field extension, as in the Main Theorem above.  The casual reader may wish to skip to the final section to see the power of the Ohm-Rush content abstraction.

\section{A refresher on the Ohm-Rush content function}\label{sec:OR}
\begin{defn}
Let $R$ be a ring, $S$ an $R$-algebra, and $M$ an $R$-module.  For $f\in M$, the (Ohm-Rush) \emph{content} $\orc(f)$ of $f$ is defined (\cite{OhmRu-content}; current terminology from \cite{nmeSh-OR}) as the intersection of all ideals $I$ of $R$ such that $f\in IM$.  When we want to specify the ring, we may write $\orc_R$; to specify the ring and the module, we write $\orc_{RM}$. We say that $M$ is an \emph{Ohm-Rush module} if $f\in \orc(f)M$ for all $f\in M$.  If $S$ is Ohm-Rush as an $R$-module, we call it an \emph{Ohm-Rush algebra}.

If an Ohm-Rush $R$-algebra $S$ has the property that $\orc(f)\orc(g) \subseteq \sqrt{\orc(fg)}$ for all $f,g \in S$ (or equivalently, for all $\p \in \Spec R$ with $\p S \neq S$, one has $\p S \in \Spec S$), then we say $S$ is a \emph{weak content algebra} over $R$ \cite{Ru-content}.

If an Ohm-Rush $R$-algebra $S$ is faithfully flat and for any multiplicative subset $W$ of $R$ and any $f,g \in S$ with $\orc(f)_W = R_W$, we have $\orc(fg)_W = \orc(g)_W$, we say $S$ is a \emph{semicontent algebra} over $R$ \cite{nmeSh-OR}.

If an Ohm-Rush $R$-algebra $S$ is faithfully flat and for any $f,g \in S$, there is some positive integer $k$ with $\orc(f)^k \orc(g) = \orc(f)^{k-1}\orc(fg)$, we say $S$ is a \emph{content algebra} over $R$ \cite[\S6]{OhmRu-content}.

If an Ohm-Rush $R$-algebra $S$ has the property that for any $f, g \in S$, we have $\orc(f)\orc(g) = \orc(fg)$, we say the $R$-algebra $S$ is \emph{Gaussian} \cite{Nas-ABconj}.

If $f\in S$ is such that $\orc(fg) = \orc(f)\orc(g)$ for all $g\in S$, we say $f$ is a \emph{Gaussian element} of the algebra.
\end{defn}

Recall also that for algebras $R \rightarrow S$, \[
\text{faithfully flat Gaussian} \implies \text{content} \implies \text{semicontent} \implies \text{weak content}
.
\]

Other facts we will use without further comment (the first two of which follow immediately from the definitions) are collected in the following statement:
\begin{prop}\label{pr:omnibus}
Let $R$ be a ring, $M$ an Ohm-Rush $R$-module, $I$ an ideal of $R$, and $S$ an Ohm-Rush $R$-algebra.
\begin{enumerate}[label=(\alph*)]
\item 
For any $f,g \in M$, we have $\orc(f+g) \subseteq \orc(f)+\orc(g)$.
\item 
For any $f, g\in S$, we have $\orc(fg) \subseteq \orc(f)\orc(g)$.
\item 
\cite[Corollary 1.6]{OhmRu-content} If $M$ is flat over $R$, $a\in R$, and $f\in M$, we have $\orc(af) = a\orc(f)$.
\item 
\label{it:loc} \cite[Theorem 3.1]{OhmRu-content} If $M$ is flat over $R$ and $W$ is a multiplicative subset of $R$, then $M_W$ is an Ohm-Rush $R_W$-module, and for any $f\in M$ and $t\in W$, we have $\orc_{R_W M_W}(f/t) = (\orc_{RM}(f))_W$.
\item\label{it:modout} \cite[Remark 2.3(d)]{OhmRu-content} $M/IM$ is an Ohm-Rush $R/I$-module, and for any $f\in M$, we have $\orc_{R/I, M/IM}(f+IM) = (\orc_{RM}(f)+I)/I$ as ideals of $R/I$
\end{enumerate}
\end{prop}

\section{Gaussian elements and locally principal content}\label{sec:Gausslp}

It is well known that a polynomial with invertible content ideal is Gaussian.  In fact more is true, and we can make the following quite general statement by using the Ohm-Rush context.

\begin{thm}\label{thm:invGauss}
Let $R$ be a commutative ring and let $S$ be a semicontent $R$-algebra.  Let $f\in S$ be an element such that $\orc(f)$ is locally principal.  Then $f$ is Gaussian.
\end{thm}

\begin{proof}
We may assume by Proposition~\ref{pr:omnibus}\ref{it:loc} that $(R,\m)$ is local.  Then there is some $a\in R$ such that $\orc(f)=aR$.  If $a=0$, then by the Ohm-Rush property $f=0$ and clearly the zero element is Gaussian.  Thus, we may assume that $a\neq 0$.  In this case, we have by the Ohm-Rush property that $f \in \orc(f)S = aS$, so $f=af_0$ for some $f_0 \in S$.  By flatness, we have $aR =\orc(f) = \orc(af_0) = a\orc(f_0)$.  Thus, there is some $t\in \orc(f_0)$ with $a=at$, so that $(1-t)a=0$.  It follows that $1-t$ is a zero divisor, hence a nonunit, so that $1-t \in \m$.  Hence $t^{-1}\in R$, so that $1=t^{-1}t \in \orc(f_0)$, whence $\orc(f_0) = R$.  Then for any $g\in S$, the semicontent property gives that $\orc(g) = \orc(f_0 g)$.  Putting it all together, we have \[
\orc(f)\orc(g) = a\orc(g) = a\orc(f_0g) = \orc(af_0g) = \orc(fg),
\]
as was to be shown.
\end{proof}

The goal of the rest of this section is to exhibit general contexts in which the converse to Theorem~\ref{thm:invGauss} holds.  To this end, we first note a fact about ascent of Gorensteinness along well-behaved extensions.
\begin{lemma}\label{lem:gorpreserve}
Let $(R,\m) \rightarrow (S,\n)$ be a flat local homomorphism with $\n=\m S$ and $R$ artinian Gorenstein.  Then $S$ is also artinian Gorenstein.
\end{lemma}

\begin{proof}
First note that $S$ is zero-dimensional.  This follows from the fact that every element of $\n=\m S$ is nilpotent, so it must be the nilradical of $S$, and so since it is maximal, it must be the only prime ideal of $S$.  Moreover, $\m S$ is finitely generated (by the images of the generators of $\m$ in $S$), so by Cohen's theorem \cite[Theorem 3.4]{Mats}, $S$ is Noetherian, and hence artinian since it is zero-dimensional.

As for Gorensteinness, it is enough to show that the socle of $S$ is isomorphic to $S/\m S$.  We have \begin{align*}
\Soc S &\cong \Hom_S(S/\m S, S) \cong \Hom_S(S \otimes_R R/\m, S) \\
&\cong \Hom_R(R/\m, \Hom_S(S,S))
 \cong \Hom_R(R/\m, R \otimes_R S) \\
 &\cong \Hom_R(R/\m, R) \otimes_R S \cong (\Soc R) \otimes_R S \cong R/\m \otimes_R S \cong S/\m S.
\end{align*}
Only two of the isomorphisms above need comment.  The fourth to last follows from \cite[Theorem 3.2.14]{EnJe-relhombook}, and the second to last holds because $R$ is artinian Gorenstein.  Hence $S$ is Gorenstein.
\end{proof}

\begin{rem}\label{rem:samegens}
Let $(R,\m) \rightarrow (S,\n)$ be a flat homomorphism of local rings with $\n =\m S$ and $J$ a finitely generated ideal of $R$.  It is well-known that $\mu_R(J) = \mu_S(JS)$, where $\mu$ means minimal number of generators.  To see this, one passes by Nakayama's lemma to lengths of the modules $J/\m J$ and $JS / \n JS = JS / \m JS$ over $R$, $S$ respectively, and then one notes that for any finite-length $R$-module $M$, the length of $M$ over $R$ equals the length of $S \otimes_R M$ over $S$ by induction, flatness, and the fact that $S \otimes_R R/\m \cong S/\n$.  We will use this multiple times in the sequel.
\end{rem}

We now come to the next core theorem of the paper.  It says that in semicontent algebras over an approximately Gorenstein base ring, Gaussian elements have locally principal content.  First we recall the definition of approximate Gorensteinness.

\begin{defn}\label{def:appGor}
A Noetherian local ring $(R,\m)$ is \emph{approximately Gorenstein} if for any positive integer $N$, there is an $\m$-primary ideal $I$ such that $I \subseteq \m^N$ and $R/I$ is a Gorenstein ring.

A ring $R$ is \emph{locally approximately Gorenstein} if $R_\m$ is approximately Gorenstein for every maximal ideal $\m$ of $R$.
\end{defn}

\begin{thm}\label{thm:appGor}
Let $R$ be a locally Noetherian and locally approximately Gorenstein ring.  Let $S$ be a faithfully flat Ohm-Rush $R$-algebra such that every maximal ideal of $R$ extends to a prime ideal of $S$.  Then for any $f\in S$ that is Gaussian over $R$, $\orc(f)$ is locally principal.
\end{thm}

\begin{proof}
We begin by assuming that $(R,\m)$ is artinian, local, and Gorenstein.  By 
\cite[Theorem 3.16]{nmeSh-OR2}, $T=S_{\m S}$ is a content $R$-algebra.  By Lemma~\ref{lem:gorpreserve}, $T$ is artinian and Gorenstein.  Next, recall \cite[Theorem 6.2]{OhmRu-content} that $\orc_{RS}(f) = \orc_{RT}(f)$.  Letting $\orc = \orc_{RT}$, we claim that $\ann_T(f) \subseteq \ann_T(\orc(f)T)$.  To see this, let $g\in \ann_T(f)$.  Then $fg=0$, so since $f$ is Gaussian, we have $0=\orc(fg) = \orc(f)\orc(g)$, whence $\orc(g) \subseteq \ann_R(\orc(f))$.  But then $g \in \orc(g)T \subseteq \ann_R(\orc(f))T \subseteq \ann_T(\orc(f)T)$.  Hence $\ann_T(f) \subseteq \ann_T(\orc(f)T)$.  On the other hand, $f \in \orc(f)T$, so $\ann_T(\orc(f)T) \subseteq \ann_T(f)$.  It follows that $\ann_T(fT) = \ann_T(\orc(f)T)$.  Since $T$ is artinian Gorenstein, every ideal of $T$ is the annihilator of its annihilator \cite[Exercise 3.2.15]{BH}.  Thus we have \[
fT = \ann_T(\ann_T(fT)) = \ann_T(\ann_T(\orc(f)T)) = \orc(f)T.
\]
In particular, $\orc(f)T$ is a principal ideal of $T$.  But then by Remark~\ref{rem:samegens}, $\orc(f)$ is a principal ideal of $R$.

In the general case, by \cite[Theorem 3.1]{OhmRu-content}, we may immediately assume $(R,\m)$ is Noetherian, local, and approximately Gorenstein, with $\m S \in \Spec S$.  By the Artin-Rees lemma, there is some positive integer $n$ with $\m^n \cap \orc(f) \subseteq \m \orc(f)$.  By approximate Gorensteinness, we may choose an $\m$-primary ideal $I$ with $R/I$ Gorenstein and $I \subseteq \m^n$.  Since $S/IS$ is a faithfully Ohm-Rush algebra over $R/I$, and since the unique prime ideal of $R/I$ extends to a prime of $S/IS$, $S/IS$ is a faithfully flat weak content algebra over $R/I$.  Furthermore, we have $I \cap \orc(f) \subseteq \m^n \cap \orc(f) \subseteq \m\orc(f)$ by the above, which means that (using $\bar{ }$ to denote modding out by $I$) $\mu_R(\orc(f)) = \mu_{R/I}(\bar \orc(\bar f))$ by Proposition~\ref{pr:omnibus}\ref{it:modout}.  Moreover, $f+IS$ is Gaussian over $R/I$, since for any $g+IS \in S/IS$, we have the equalities \begin{align*}
\bar \orc(f+IS) \bar \orc(g+IS) &= \frac{\orc(f)+I}{I}\cdot\frac{\orc(g)+I}{I} = [\orc(f)\orc(g)+I]/I\\
&=[\orc(fg)+I]/I = \bar\orc(fg+IS) = \bar\orc((f+IS)(g+IS)).
\end{align*}
But by the artinian Gorenstein case we showed above, $\bar \orc(f+IS)$ is a principal ideal of $R/I$, whence by the above, $\mu_R(\orc(f))=\mu_{R/I}(\bar \orc(f+IS)) = 1$, which means that $\orc(f)$ is a principal ideal of $R$.
\end{proof}

Note that the above theorem already covers many cases.  For instance \cite[Theorem 1.7, Corollary 2.2, Remark 4.8b, and Proposition 5.6]{Ho-purity}, the following sorts of Noetherian rings $R$ are (locally) approximately Gorenstein: \begin{itemize}
\item $R$ is Cohen-Macaulay of dimension at least two.
\item $R$ is a one-dimensional domain that is a homomorphic image of a Gorenstein ring.
\item $R$ is Gorenstein.
\item $R$ is excellent and reduced.
\item $R$ is analytically unramified.
\item $R$ is complete, local, and reduced.
\item $R=A[x,y]$ or $A[\![x,y]\!]$, where $A$ is any Noetherian ring and $x,y$ are (analytic) indeterminates over $A$.
\end{itemize}

On the other hand, the result is sharp in the artinian case, in the sense that for an artinian non-Gorenstein ring, this result is typically false, as follows:

\begin{prop}\label{pr:nongor}
Let $(R,\m)$ be a local artinian ring that is not Gorenstein, and let $S$ be a semicontent $R$-algebra 
such that every 2-generated ideal of $R$ is the content of some element of $S$.  Then there is some Gaussian element of $S$ with non-principal content.
\end{prop}

\begin{proof}
Let $x,y$ be linearly independent elements of the socle of $R$.  Choose $f\in S$ such that $\orc(f) = (x,y)$.  This latter ideal is minimally 2-generated, so it is not principal.  On the other hand, $f$ is Gaussian.  To see this, let $g\in S$.  If $\orc(g) =R$, then by the semicontent property, $\orc(fg) = \orc(f) = \orc(f)\orc(g)$.  If on the other hand $\orc(g) \subseteq \m$, then $\orc(fg) \subseteq \orc(f)\orc(g) \subseteq \m\cdot(x,y) = 0$, so $\orc(fg) = \orc(f)\orc(g)$ in this case as well.
\end{proof}

Next we pass to complete local rings.  For this, we need a lifting result.

\begin{lemma}\label{lem:modout2}
Let $\alpha: R \rightarrow S$ be a flat Ohm-Rush map, and let $\{J_n\}_{n\in \N}$ be a sequence of ideals of $R$ such that $\bigcap_n J_n = 0$.  Then for any $f, g\in S$, $\orc(fg) = \orc(f)\orc(g)$ if and only if $\orc_n(\overline{fg}) = \orc_n(\bar f) \orc_n(\bar g)$ for every $n\in \N$, where $\orc$ is the content map from $\alpha$ and $\orc_n$ is the content map from the induced map $\alpha_n: R/J_n \rightarrow S/J_n S$.
\end{lemma}

\begin{proof}
Each map $\alpha_n$ is itself a flat Ohm-Rush map.  Then if $\orc(fg) = \orc(f)\orc(g)$, we have the following: \begin{align*}
\orc_n(\overline{fg}) &= (\orc(fg) + J_n)/J_n = (\orc(f)\orc(g) + J_n)/J_n \\
&= \frac{\orc(f)+J_n}{J_n} \cdot \frac{\orc(g)+J_n}{J_n} = \orc_n(\bar f) \orc_n(\bar g).
\end{align*}
On the other hand, suppose $\orc(fg) \neq \orc(f)\orc(g)$.  Since the forward containment always holds, we have $\orc(fg) \subsetneq \orc(f)\orc(g)$.  Choose $r\in \orc(f)\orc(g) \setminus \orc(fg)$.  Since $\bigcap_n J_n = 0$, there is some $n$ with $r\notin J_n$, and hence $\orc(fg) + J_n \subsetneq \orc(f)\orc(g)+J_n$.  Hence, \begin{align*}
\orc_n(\overline{fg}) &= (\orc(fg)+J_n) / J_n \neq (\orc(f)\orc(g)+ J_n)/J_n \\
&= \frac{\orc(f)+J_n}{J_n} \cdot \frac{\orc(g)+J_n}{J_n} = \orc_n(\bar f) \orc_n(\bar g). \qedhere
\end{align*}
\end{proof}

\begin{prop}\label{pr:gausscomplete}
Let $(R,\m)$ be a Noetherian local ring, and let $S$ be a
n $R$-algebra that is free as an $R$-module.  Let $f\in S$.  Then $f$ is Gaussian over $R$ if and only if the image $f \otimes 1$ of $f$ in $S \otimes_R \hat R$ is Gaussian over $\hat R$, where $\hat R$ is the $\m$-adic completion of $R$.
%
\end{prop}

\begin{proof}
The hypotheses imply that $S$ is a faithfully flat Ohm-Rush algebra over $R$.  Let $\{e_\lambda\}_{\lambda \in \Lambda}$ be a basis for $S$ as a free $R$-module. Let $\alpha: R \rightarrow S$ be the map that exhibits $S$ as an $R$-algebra, and let $\alpha_n: R/\m^n \rightarrow S/\m^n S$ be the induced maps on quotients for all $n\in \N$.  Let $f\in S$. Then by Lemma~\ref{lem:modout2} and the Krull intersection theorem, $f$ is Gaussian with respect to $\alpha$ if and only if for each $n\in \N$, $\bar f_n := f+\m^n S$ is Gaussian with respect to $\alpha_n$.

On the other hand, note that the map $\beta: \hat R \rightarrow S \otimes_R \hat R$ induced from $\alpha$ is an Ohm-Rush map, since $S \otimes_R \hat{R}$ is a free module over $\hat R$ with a basis indexed by $\Lambda$.  Then by another application of Lemma~\ref{lem:modout2}, $f \otimes 1$ is Gaussian with respect to $\beta$ if and only if $\overline{(f \otimes 1)}_n := (f \otimes 1) + \hat\m^n(S \otimes_R \hat R)$ is Gaussian with respect to the map $\beta_n: \hat R / \hat{\m}^n \rightarrow (S \otimes_R \hat R) / \hat{\m}^n (S \otimes_R \hat R)$ induced from $\beta$ for all $n \in \N$.  

However, $\beta_n$ is canonically identified with $\alpha_n$ via the isomorphism $\hat R / \hat \m^n \cong R/\m^n$ and the following identifications: \begin{align*}
\frac{S \otimes_R \hat R}{\hat m^n (S \otimes_R \hat R)} &\cong S \otimes_R \hat{R} / \hat\m^n = \left(\bigoplus_{\lambda \in \Lambda} R e_\lambda\right) \otimes_R \hat R / \hat\m^n \\
&\cong \bigoplus_{\lambda \in \Lambda} (\hat R / \hat \m^n)e_\lambda \cong \bigoplus_{\lambda \in \Lambda} (R/\m^n) e_\lambda \\
&\cong \left(\bigoplus_{\lambda \in \Lambda} R e_\lambda\right) \otimes_R R/\m^n = S \otimes_R R/\m^n \cong S/\m^n S.
\end{align*}
Summarizing, we have for any $f\in S$:

 \begin{align*}
 f \text{ is Gaussian w.r.t. }\alpha &\iff \bar f_n \text{ is Gaussian w.r.t. } \alpha_n \text{ for all } n\\
 &\iff \overline{(f\otimes 1)}_n \text{ is Gaussian w.r.t. } \beta_n \text{ for all }n\\
 &\iff f \otimes 1 \text{ is Gaussian w.r.t. } \beta. \qedhere
%
 \end{align*}
 \end{proof}
 
 Next we note an analogue of McCoy's theorem (see \cite[Theorem 2]{Mcc-divzero}).
 \begin{prop}\label{pr:McCoy}
Let $R$ be a ring and $S$ a semicontent $R$-algebra.  Let $f\in S$.  If $\orc(f)$ contains a regular element of $R$, then $f$ is a regular element of $S$.  The converse holds if $R$ is Noetherian.
 \end{prop}
 
 \begin{proof}
 First suppose $\orc(f)$ contains a regular element $w$.  Let $g\in S$ with $fg=0$.  Let $W$ be the multiplicative subset of $R$ generated by $w$.  Then $\orc(f)_W = R_W$, so by definition of the semicontent property, $\orc(g)_W = \orc(fg)_W = 0$.  Since $\orc(g)$ is finitely generated, it follows that there is some $n\in \N$ with $w^n\orc(g)=0$.  But then since $w^n$ is a regular element of $R$, we have $\orc(g)=0$, whence $g=0$.
Thus, $f \in S$ is regular.

 On the other hand, suppose $R$ is Noetherian and $\orc(f)$ consists of zero divisors.
 By \cite[Theorem 82]{Kap-CR}, there is some nonzero $r\in R$ with $0=r\orc(f) = \orc(rf)$.  Hence $rf=0$, and so $f$ is not a regular element of $S$.
 \end{proof}
 
 \begin{thm}\label{thm:completefree}
 Let $(R,\m)$ be a complete Noetherian local ring.  Let $S$ be a Noetherian semicontent $R$-algebra that is free as an $R$-module.  Let $f\in S$ be a regular element that is Gaussian over $R$.  Then $\orc(f)$ is principal.
 \end{thm}
 
 \begin{proof}
Since $R$ is principal, we may assume $\orc(f) \subseteq \m$.  Let $\n$ be the nilradical of $R$.  Then $S/\n S$ is Ohm-Rush over $R/\n$ and $\orc(f+\n S) = (\orc(f)+\n)/\n$.  But since the maximal ideal of $R/\n$ extends to a prime ideal of $S/\n S$ and $R/\n$ is approximately Gorenstein \cite[1.4 and 1.7]{Ho-purity}, we have by Theorem~\ref{thm:appGor} that $\orc(f+\n S) = (\orc(f)+\n)/\n$ is a principal ideal of $R/\n$.  Choose an element $a'\in R$ that generates $\orc(f)$ mod $\n$.  By Proposition~\ref{pr:McCoy}, $a'$ must be a regular element of $R$.  By the freeness condition, there exist $e_1, \ldots, e_n$ in a free basis of $S$ over $R$ such that $f = \sum_{i=1}^n a_i e_i$ for uniquely determined $a_i \in R$, and then $\orc(f) = (a_1, \ldots, a_n)$ \cite[Corollary 1.4]{OhmRu-content}.  But then likewise, working over $R/\n$, we have $(\overline{a'}) = \bar\orc(f+\n S) = (\orc(f)+\n) / \n = (\overline{a_1}, \ldots, \overline{a_n})$.  Since $R/\n$ is local, any minimal system of generators of $(\overline{a'})$ must be of size 1.  Hence there is some $k$, $1\leq k \leq n$, with $a_k \equiv a'$ mod $\n$.  Write $a=a_k$.  Then $a$ is a regular element of $R$, $\orc(f)$ is generated mod $\n$ by the image of $a$, and $a\in \orc(f)$.

Now let $T$ be the total quotient ring of $R$ and set $I_j := \n^j T \cap R$ for positive integers $j$.  Note that $\n = I_1$, $I_j^2 \subseteq I_{j+1}$, and $a$ is regular on $R/I_j$ for all $j$.

We claim that $\orc(f) \subseteq aR + I_j$ for all $j$.  Since $aR \subseteq \orc(f)$ and $I_j = 0$ for $j\gg 0$, the claim will prove the result.  We prove the claim by induction, starting with the observation that $\orc(f)\subseteq aR+\n = aR + I_1$.  Now let $j\geq 1$ and assume that $\orc(f) \subseteq aR + I_j$.  For each $i\in \{1, \ldots, n\}$, we have $a_i \in aR+I_j$, so we write $a_i = r_i a + c_i$, $r_i \in R$, $c_i \in I_j$.  For $i=k$, we may do this with $r_k=1$ and $c_k=0$.  Write $f_1 := \sum_{i=1}^n r_i e_i$ and $f_2 := \sum_{i=1}^n c_i e_i$.  Note that $f_2 \in I_j S$.  Then $f=af_1 + f_2$.  Since $\orc(f) \equiv \orc(f_2)$ mod $aR$ and $a\in \orc(f)$, we have $\orc(f) = aR + \orc(f_2)$.  Now set $g := af_1 - f_2$.  We similarly have $a\in \orc(g)$ (since $a$ is the coefficient of $e_k$ in $g$) and $\orc(g) \equiv \orc(f_2)$ mod $aR$, so $\orc(g) = aR + \orc(f_2)$.  But then by the Gaussian property for $f$, we have \begin{align*}
a^2 R + a\orc(f_2) + \orc(f_2)^2 &= (aR + \orc(f_2))^2 = \orc(f)\orc(g) \\
&= \orc(fg) = \orc(a^2 f_1^2 - f_2^2) \subseteq a^2 \orc(f_1^2) + \orc(f_2^2) \\
&\subseteq a^2 R + \orc(f_2)^2 \subseteq a^2 R + I_j^2 \subseteq a^2 R + I_{j+1}.
\end{align*}
Since $\orc(f_2)^2 \subseteq I_j^2 \subseteq I_{j+1}$, it follows that $a\orc(f_2) \subseteq a^2 R + I_{j+1}$.  But since $a$ is regular mod $I_{j+1}$, it follows that $\orc(f_2) \subseteq aR + I_{j+1}$.  Thus, $\orc(f) = aR + \orc(f_2) \subseteq aR + I_{j+1}$, completing the inductive step, and thus the proof.
 \end{proof}

 \begin{thm}\label{thm:localNoeth}
 Let $R$ be a locally Noetherian (resp. Noetherian) ring, and let $S$ be an $R$-algebra that is free as an $R$-module, such that $S \otimes_R \widehat{R_\m}$ is a 
semicontent $\widehat{R_\m}$-algebra for all maximal ideals $\m$ of $R$.  Let $f\in S$ be a regular element that is Gaussian with respect to the given map $R \rightarrow S$.  Then $\orc(f)$ is locally principal (resp. invertible).
 \end{thm}

 \begin{proof}
 Let $\m$ be a maximal ideal of $R$.  Then $f/1 \in S_W$ is Gaussian over $R_\m$, where $W = R \setminus \m$.  Then by Proposition~\ref{pr:gausscomplete}, $(f/1) \otimes 1 \in S_W \otimes_{R_\m} \widehat{R_\m} = S \otimes_R R_\m \otimes_{R_\m} \widehat{R_\m} =S \otimes_R \widehat{R_\m}$ is Gaussian over $\widehat{R_\m}$.  Moreover, since $\widehat{R_\m}$ is flat over $R_\m$ and $R_\m$ is flat over $R$, $f \otimes 1$ is regular in $S \otimes_R \widehat{R_\m}$.   Then Theorem~\ref{thm:completefree} applies to show that $\orc((f/1) \otimes 1)$ is principal in $\widehat{R_\m}$.  But $\orc((f/1) \otimes 1) = \orc(f/1)\widehat{R_\m}$ by the freeness assumption, so by 
 Remark~\ref{rem:samegens}, 
$\orc(f/1)$ is principal in $R_\m$. But $\orc(f/1) = \orc(f)R_\m$ by \cite[Theorem 3.1]{OhmRu-content}, so $\orc(f)$ is locally principal.

In the case where $R$ is Noetherian, we have by Proposition~\ref{pr:McCoy} that $\orc(f)$ contains a regular element, but then by \cite[Teorema 1.1]{Gr-fracinv}, $\orc(f)$ is invertible. 
 \end{proof}
 
 \begin{cor}\label{cor:localNoeth}
 Let $R$ be a locally Noetherian (resp. Noetherian) ring. Let $S$ be an $R$-algebra satisfying the conditions of Theorem~\ref{thm:localNoeth}.  Then a regular element $f\in S$ is Gaussian over $R$ if and only if $\orc(f)$ is locally principal (resp. invertible).
 \end{cor}
 
 \begin{proof}

Theorems~\ref{thm:invGauss} and \ref{thm:localNoeth}
.
 \end{proof}

\section{A characterization of Pr\"ufer domains}\label{sec:Pr}

For the next result, we use some ideas from Hwa Tsang's thesis \cite{Ts-Gauss}
in the proof.

\begin{prop}\label{pr:nonGauss}
Let $(R,\m)$ be a local integral domain and let $S$ be an $R$-algebra that is free of rank $>1$ as an $R$-module.  Suppose there is an ideal $I$ of $R$ that is minimally generated by two elements.  Then there is some non-Gaussian element $f\in S$ with $c(f)=I$.
\end{prop}

\begin{proof}
Write $I=(a,b)$, and let $s,t$ be part of a free basis for $S$ over $R$.  Let $f=as+bt$; then $\orc(f)=I$ by \cite[Corollary 1.4]{OhmRu-content}.

Suppose $f$ is Gaussian.  Let $g=bs+at$.  Then \[
(a,b)^2 = \orc(f)\orc(g) = \orc(fg) = \orc(abs^2 +(a^2 + b^2)st + abt^2) \subseteq (ab, a^2+b^2),
\]
where the containment follows from flatness of $S$ over $R$.  
In particular, $a^2 \in (ab, a^2+b^2)$.  Say $a^2 = \alpha ab + \beta (a^2+b^2)$ where $\alpha, \beta \in R$.  Then $(1-\beta)a^2 - \beta b^2 = \alpha ab$, so since $R$ is local, either $\beta$ or $1-\beta$ is a unit.  In the former case, $b^2 \in (ab, a^2)$, whereas in the latter case, $a^2 \in (ab, b^2)$.  Thus, by switching labels of $a,b$ if necessary, we may assume $a^2 \in (ab, b^2)$.

Accordingly write $a^2 = \gamma ab + \delta b^2$, where $\gamma, \delta \in R$.  Let $h=(a-\gamma b)s - bt$.  Then \begin{align*}
(a,b)^2 &= \orc(f)\orc(h) = \orc(fh) = \orc((a^2 - \gamma ab)s^2-\gamma b^2 st- b^2 t^2) \\
&= \orc(\delta b^2 s^2 - \gamma b^2 st - b^2 t^2) = \orc(b^2(\delta s^2 - \gamma st - t^2)) \subseteq b^2 R,
\end{align*}
where the containment again follows from flatness of $S$ over $R$.  
In particular, $ab \in b^2R$.  But then since $b$ is regular, we have $a\in bR$, whence $(a,b) = (b)$.  But this contradicts the fact that $I$ is not principal.  Hence $f$ cannot be Gaussian.
\end{proof}

As a corollary, we obtain the following criterion for an integral domain to be Pr\"ufer (and hence Dedekind if we also assume the ring is Noetherian).
\begin{thm}\label{thm:Pr}
Let $R$ be an integral domain.  Suppose there exists a Gaussian $R$-algebra $S$ that is free of rank $>1$ (i.e. free and a nontrivial extension) as an $R$-module.  Then $R$ is a Pr\"ufer domain.
\end{thm}

\begin{proof}
Since all issues are local \cite[Theorem 3.1]{OhmRu-content}, we may assume $(R,\m)$ is a local domain, so that we need only show that every 2-generated ideal is principal.  But this follows directly from Proposition~\ref{pr:nonGauss}.
\end{proof}

\section{Applications}\label{sec:apps}

As the theorems in $\S\ref{sec:Gausslp}$ and $\S\ref{sec:Pr}$ are generalizations of theorems known in polynomial extensions by a single variable, they have applications insofar as there exist ring extensions other than such polynomial extensions that satisfy the conditions.  For example, we may extend some of the results of \cite{HeiHu-Gauss} to the multivariate case, and beyond.  In particular, our next two results extend \cite[Theorem 1.5]{HeiHu-Gauss} and part of \cite[Theorem 3.3]{HeiHu-Gauss} respectively.
\begin{thm}\label{thm:appGorMonoid}
Let $R$ be a locally Noetherian and locally approximately Gorenstein ring, and $G$ a torsion-free commutative cancellative monoid.  Then for any $f \in R[G]$ that is Gaussian over $R$, $\orc(f)$ is locally principal.
In particular this holds when $G=\N^n$, so that $R[G] = R[x_1, \ldots, x_n]$.
\end{thm}
\begin{proof} Use \cite[Theorem 3]{No-content} to apply Theorem~\ref{thm:appGor} to $S=R[G]$. \end{proof}

\begin{thm}\label{thm:regMonoid}
Let $R$ be a locally Noetherian (resp. Noetherian) ring, let $G$ be a torsion-free commuative cancellative monoid, and let $f \in R[G]$ be a regular element that is Gaussian over $R$.  Then $\orc(f)$ is locally principal (resp. invertible).
In particular, this holds when $G=\N^n$, so that $R[G] = R[x_1, \ldots, x_n]$.
\end{thm}

\begin{proof} Use \cite[Theorem 3]{No-content} to apply Theorem~\ref{thm:localNoeth} to $S=R[G]$. \end{proof}


For another area of application, recall \cite[Theorems 4.3 and 3.10]{nmeSh-OR} that when $R$ is Noetherian, $R[\![x_1, \ldots, x_n]\!]$ is a semicontent $R$-algebra, even though it is rarely free as an $R$-module.  Hence, we have the following.

\begin{thm}\label{thm:appGorpower}
Let $R$ be a Noetherian locally approximately Gorenstein ring, and $n\in \N$.  Then for any Gaussian $f\in R[\![x_1, \ldots, x_n]\!]$, $\orc(f)$ is locally principal.
\end{thm}
\begin{proof}
Apply Theorem~\ref{thm:appGor} with $S=R[\![x_1, \ldots, x_n]\!]$.
\end{proof}

Next, recall the following result from \cite[Proposition 3.8]{nmeSh-OR2}:

\begin{prop}\label{pr:fieldext}
Let $L/K$ be a separable field extension such that $K$ is algebraically closed in $L$.  Then for any $K$-algebra $A$, $B = A \otimes_K L$ is a module-free weak content algebra over $K$.  If $A = K[x_1, \ldots, x_m]$ is a polynomial ring over $K$ and $B = L[x_1, \ldots, x_n]$ where $n\geq m$, then $B$ is a module-free semicontent $A$-algebra.
\end{prop}

To determine the content of an element $f \in L[x_1, \ldots, x_m]$ with respect to $R=K[x_1, \ldots, x_m]$, where $L/K$ is \emph{any} field extension, fix a vector space basis $\mathcal B$ for $L$ over $K$; then represent each coefficient of a monomial in $f$ as a $K$-linear combination of the elements of $\mathcal B$.  Collecting terms, you then get an $R$-linear combination of elements of $\mathcal B$.  Collect together the corresponding coefficients and see what ideal of $R$ they generate.  This is then the content of $f$ with respect to $R$.  

\begin{cor}
Let $(A,\m)$ be a local integral domain containing a field $K$, such that $A$ is not a valuation ring, and let $L/K$ be a nontrivial field extension.  Then $B = A \otimes_K L$ is not Gaussian over $A$.  In fact, for any minimally 2-generated ideal $I$ in $A$, there is some non-Gaussian $f\in B$ whose content over $A$ is $I$.
\end{cor}

\begin{proof}
This follows from Theorem~\ref{thm:Pr} and Proposition~\ref{pr:nonGauss}.
\end{proof}

On the other hand, we have the following 
result, which can be seen as a characterization of Gorensteinness among finitely generated Artinian $K$-algebras:

\begin{cor}\label{cor:fieldextGor}
Let $A=K[x_1,\ldots, x_m]$, let $J$ be a finite colength ideal of $A$, let $L/K$ be a nontrivial separable field extension such that $K$ is algebraically closed in $L$, and let $B := L[x_1, \ldots, x_n]$ for some $n\geq m$. Setting $R := A/J$ and $S := B/JB$, the following are equivalent: \begin{enumerate}
\item[(a)] For any $f\in S$ that is Gaussian over $R$, $\orc_{RS}(f)$ is locally principal.
\item[(b)] $R$ is Gorenstein.
\end{enumerate}
\end{cor}

\begin{proof}
(b) $\implies$ (a) by Theorem~\ref{thm:appGor}, whereas (a) $\implies$ (b) by Proposition~\ref{pr:nonGauss} and the contrapositive to Proposition~\ref{pr:nongor}.
\end{proof}

Along the same vein, we have the following:
\begin{cor}\label{cor:fieldextGauss}
Let $R$ be a Noetherian $K$-algebra, where $K$ is a field, and let $L/K$ be a separable field extension such that $K$ is algebraically closed in $L$.  Let $S := R \otimes_K L$, and let $f\in S$ be a regular element.  Then $f$ is Gaussian over $R$ iff $\orc_{RS}(f)$ is locally principal.
\end{cor}

\begin{proof}
It is clear that $S$ is free as an $R$-module.  Moreover, for any maximal ideal $\m$ of $R$, $\widehat{R_\m} \otimes_R S \cong \widehat{R_\m} \otimes_K L$, which by the argument in the proof of \cite[Proposition 3.8]{nmeSh-OR2} is a semicontent $\widehat{R_\m}$-algebra.  Then Corollary~\ref{cor:localNoeth} finishes the proof. 
\end{proof}

\section*{Acknowledgments}
Thanks to Peyman Nasehpour for pointing out a simplification in \S\ref{sec:Gausslp}.
 
\providecommand{\bysame}{\leavevmode\hbox to3em{\hrulefill}\thinspace}
\providecommand{\MR}{\relax\ifhmode\unskip\space\fi MR }
\providecommand{\MRhref}[2]{%
  \href{http://www.ams.org/mathscinet-getitem?mr=#1}{#2}
}
\providecommand{\href}[2]{#2}

\end{document}